\documentclass[submission]{FPSAC2019}
\usepackage{tikz,url}
\makeatletter\let\@firsthead\empty\makeatother

\newtheorem{thm}{Theorem}

\newtheorem{lem}{Lemma}

\newtheorem{example}{Example}

\def\O{\mathrm{O}}
\def\sgn{\operatorname{sgn}}
\def\diag{\operatorname{diag}}

\def\bigtimes{\mathop{\hbox{\Large$\times$}}}

\makeatletter
\def\testb#1{\testb@i#1,,\@nil}%
\def\testb@i#1,#2,#3\@nil{%
  \draw[->] (O) --++(#1);
  \ifx\relax#2\relax\else\testb@i#2,#3\@nil\fi}
\makeatother
\newcommand{\makediag}[1]{
    \coordinate (O) at (0,0); \coordinate (N) at (0,0.8);
    \coordinate (NE) at (0.8,0.8); \coordinate (E) at (0.8,0);
    \coordinate (SE) at (0.8,-0.8); \coordinate (S) at (0,-0.8);
    \coordinate (SW) at (-0.8,-0.8);\coordinate (W) at (-0.8,0);
    \coordinate (NW) at (-0.8,0.8); \coordinate (B1) at (1.2,1.2);
    \coordinate (B2) at (-1.2,-1.2);
    \testb{#1}
}
\newcommand{\diagr}[1]{
  \begin{tikzpicture}[baseline=-3pt,scale=0.3]\makediag{#1}\end{tikzpicture}
}

\title{Inhomogeneous Restricted Lattice Walks}

\author{Manfred Buchacher\addressmark{1} \and Manuel Kauers \addressmark{1}\thanks{Both authors were supported by the Austrian FWF grant F5004.}}

\address{\addressmark{1}Institute for Algebra, Johannes Kepler University, Linz, Austria}

\received{\today}

\abstract{%
  We consider inhomogeneous lattice walk models in a half-space and in the quarter plane.
  For the models in a half-space, we show by a generalization of the kernel method to
  linear systems of functional equations that their generating functions are always algebraic.
  For the models in the quarter plane, we have carried out an experimental classification of 
  all models with small steps. We discovered many (apparently) D-finite cases for most of
  which we have no explanation yet.
}


\usepackage[backend=bibtex]{biblatex}
\begin{document}

\let\set\mathbb

\maketitle

\section{Introduction}

Given a specific counting problem, it is often easy to write down a functional equation for the corresponding
generating function, but it can be quite hard to derive from it something interesting about its solution.
Counting problems for restricted lattice walks are a source of functional equations which are neither trivial nor hopeless,
and therefore interesting. 
Using the kernel method, Banderier and Flajolet~\cite{banderier02,flajolet09} have shown that the generating functions for
lattice walks restricted to a half-space are always algebraic.
For walks restricted to the quarter plane, the functional equations are more intricate, and as a consequence, the
resulting generating functions come in more flavours: algebraic, transcendentally D-finite, non-D-finite but ADE,
not even ADE. A systematic classification was initiated in a seminal paper by Bousquet-M\'elou and Mishna~\cite{bousquet10} and
has since led to a substantial amount of literature, see~\cite{bostan17} and the comprehensive list of references given there,
as well as \cite{bernardi17,dreyfus18} for some more recent developments. 

Since the main questions raised in the paper of Bousquet-M\'elou and Mishna have been answered, the focus has shifted to
modified versions of the problem, including, for example, weighted steps~\cite{kauers15c,courtiel17}, longer
steps~\cite{melczer18}, higher dimensions~\cite{bostan16a,bacher16,bogosel18}, or walks in regions with interacting
boundaries~\cite{beaton18}. All these variations are homogeneous in the sense that the walks are formed from a fixed set
of admissible steps. Little is known about inhomogeneous models, i.e., lattice walk models where the set
of admissible steps may vary with the time and/or the position of the walk. D'Arco et al.~\cite{darco16} have studied such
a case. They determined the asymptotics for walks in the quarter plane where at position $(i,j)$, the next step is taken
from \diagr{W,N,S,E} if $i+j$ is even and from \diagr{W,N,S,E,NW,NE,SW,SE} if $i+j$ is odd. 
In the present paper, we want to study such inhomogeneous models more systematically. 

We consider half-space models and quarter plane models.
For the half-space, we show for a general class of inhomogeneities that the generating function is invariably algebraic.
This generalizes the classical result of Banderier and Flajolet mentioned above.
For the proof, we adapt an argument given by Bousquet-Mel\'ou and Jehanne~\cite{bousquet06} to the case of linear systems
of functional equations (Section~\ref{sec:3}). This result may also be useful in other contexts.
For regions defined by more than one restriction, we generalize the notion of dimension
introduced by Bostan et al.~\cite{bostan16a} to inhomogeneous models (Section~\ref{sec:4a}).
For the quarter plane (Section~\ref{sec:4}), we give an experimental classification for models with short steps and
two specific inhomogeneities. We recognize many generating functions as D-finite,
but in most cases we have no explanation for their D-finiteness. 




\section{Inhomogeneous Lattice Walks in the Half-Space}\label{sec:2}

We consider walks in a half-space $\set Z_{\geq0}\times\set Z^{d-1}$ which start at some point $(i_0,\mathbf j_0)$ of the half-space,
consist of $n$ steps, and end at some point $(i,\mathbf j)$ of the half-space.
Restrictions are imposed on the steps a walk may contain.
For a \emph{homogeneous} model, there is a fixed finite set $\mathbf{S}\subseteq\set Z^d$, called
the step set, and every step of the walk must be taken from this step set.
For an \emph{inhomogeneous} model, there are several step sets, and it may depend on the current
position $(i,\mathbf j)$ or on the number $n$ of steps taken so far from which step set the next
step must be chosen. 

We consider inhomogeneities which can be
described as follows. Let $k\geq 1$, let $\mathbf{p}=(p_1,\dots,p_k)$ be a
vector of polynomials in $\mathbb{Z}[i,j_1,\dots,j_{d-1},n]$ of total degree at most one, and
let $\mathbf{m}=(m_1,\dots,m_k)$ be a vector of positive integers. For every vector
$\mathbf{r}\in\bigtimes_{q=1}^k\{0,\dots,m_q-1\}$, let
$\mathbf{S}_{\mathbf{r}}$ be a finite subset of~$\set Z^d$. We consider walks
in the half-space $\set Z_{\geq0}\times\set Z^{d-1}$ starting at $(i_0,\mathbf j_0)$ and consisting
of steps taken from $\bigcup_{\mathbf{r}} \mathbf{S}_{\mathbf{r}}$ with the
restriction that whenever the position of the walk after $n$ steps is $(i,\mathbf j)$,
the $n+1$st step must be taken from $\mathbf{S}_{\mathbf{p}(i,\mathbf j,n)\bmod\mathbf{m}}$,
where the evaluation at $i,\mathbf j,n$ and the $\bmod$-operation in the index are meant
component-wise.

\begin{example}\label{0}
  Let $\mathbf{S}_0=\diagr{N,E,S,W}$
  and $\mathbf{S}_1=\diagr{N,E,S,W,SW,SE,NW,NE}$, 
  and consider walks in $\set Z_{\geq0}\times \set Z$ starting at $(0,0)$
  and consisting of steps taken from $\mathbf{S}_0\cup\mathbf{S}_1$
  with the restriction that whenever the current position of a walk is $(i,j)$,
  the next step must be taken from $\mathbf{S}_{i+j\bmod 2}$.
  With the terminology introduced above, we have $k=1$, $\mathbf p=i+j$, and $\mathbf m=2$.
\end{example}

When the step set depends on the current position $(i,\mathbf j)$, as in
Example~\ref{0}, we call the model \emph{space inhomogeneous.}  In the other
extreme, when the step set only depends on the time~$n$, the model would be
called \emph{time inhomogeneous.}  In general, the step set may depend on both
time and position.

Fix an inhomogeneous lattice walk model as introduced above, and for each
$\mathbf r$ fix a function $w_{\mathbf r}\colon\mathbf S_{\mathbf r}\to\set K$,
where $\set K$ is a field of characteristic zero.
We call $w_{\mathbf r}(u,\mathbf v)$ the \emph{weight}
of the step $(u,\mathbf v)\in\mathbf S_{\mathbf r}$. The weight of a walk is the product
of the weights of the steps it consists of. Note that the same step $(u,\mathbf v)$ may
have different weights depending on the time and the position at which it is taken.

For any $i,\mathbf j,n$, let $f_{i,\mathbf j,n}$ be the sum of the weights of all walks from $(i_0,\mathbf j_0)$
to $(i,\mathbf j)$ of length~$n$, and let $F(x,\mathbf y,t)=\sum_{i,\mathbf j,n}f_{i,\mathbf j,n}x^i\mathbf y^{\mathbf j}t^n\in\set K[x,\mathbf y,\mathbf y^{-1}][[t]]$ be the
corresponding generating function, where we write $\mathbf y$ and $\mathbf y^{-1}$
for $y_1,\dots,y_{d-1}$ and $y_1^{-1},\dots,y_{d-1}^{-1}$, respectively, and $\mathbf y^{\mathbf j}$
for $y_1^{j_1}\cdots y_{d-1}^{j_d}$.
Note that $f_{i,\mathbf j,n}$ is just the total number of walks in the model if all weights
are defined as~$1$. 
We have $F=\sum_{\mathbf{r}} F_{\mathbf{r}}$, where for each $\mathbf{r}$
we define $F_{\mathbf{r}}$ as the series consisting of all terms $f_{i,\mathbf j,n}x^i\mathbf y^{\mathbf j}t^n$
of $F$ for which $\mathbf{p}(i,\mathbf j,n)=\mathbf{r}\bmod\mathbf{m}$.

\begin{example}
  Continuing the previous example, $F_0$ is the generating function counting walks ending
  at a point $(i,j)$ with $i+j$ even and $F_1$ is the generating function counting walks
  ending at a point $(i,j)$ with $i+j$ odd.
  The full generating function clearly is $F=F_0+F_1$. 
\end{example}

For every $\mathbf s\in\bigtimes_{q=1}^k\{0,\dots,m_q-1\}$, define
$\mathbf S_{\mathbf{r}}^{\mathbf{s}}\subseteq\mathbf S_{\mathbf r}$ as the set
of all steps $(u,\mathbf v)\in\mathbf S_{\mathbf r}$ which are such that whenever we are at
position $(i,\mathbf j)$ at time~$n$, for some $i,\mathbf j,n$ with
$\mathbf{p}(i,\mathbf j,n)=\mathbf{r}\bmod\mathbf{m}$, taking the step $(u,\mathbf v)$ will
bring us to a position where the next step has to be selected from
$\mathbf{S}_{\mathbf s}$, i.e., such that
$\mathbf{p}(i+u,\mathbf j+\mathbf v,n+1)=\mathbf{s}\bmod\mathbf{m}$.  Note that $\mathbf
S_{\mathbf r}^{\mathbf s}$ is well-defined in the sense that it does not depend
on $(i,\mathbf j,n)$, for whenever $(i',\mathbf j',n')$ is such that
$\mathbf{p}(i,\mathbf j,n)=\mathbf{p}(i',\mathbf j',n')\bmod\mathbf m$, then also
$\mathbf{p}(i+u,\mathbf j+\mathbf v,n+1)=\mathbf{p}(i'+u,\mathbf j'+\mathbf v,n'+1)\bmod\mathbf m$ for any
$u\in\set Z$ and $\mathbf v\in\set Z^{d-1}$,
because the components of $\mathbf{p}$ are polynomials of total degree
at most~$1$.

The definition of the subsets $\mathbf S_{\mathbf r}^{\mathbf s}$ is made in
such a way that whenever we have a walk of some length $n$ ending at some
position $(i,\mathbf j)$, with $\mathbf{p}(i,\mathbf j,n)=\mathbf r\bmod\mathbf m$, then appending
any step $(u,\mathbf v)\in\mathbf S_{\mathbf r}^{\mathbf s}$ to the
walk gives a new walk of length $n+1$ which ends at position $(i+u,\mathbf j+\mathbf v)$. For the length
and endpoint of the extended walk, we have $\mathbf{p}(i+u,\mathbf j+\mathbf v,n+1)=\mathbf
s\bmod\mathbf m$. In order to respect the boundary condition, only steps
$(u,\mathbf v)\in\mathbf S_{\mathbf r}^{\mathbf s}$ with $i+u\geq0$ can be used.

Writing $S_{\mathbf r}^{\mathbf s}:=\sum_{(u,\mathbf v)\in\mathbf{S}_{\mathbf r}^{\mathbf s}}
w_{\mathbf r}(u,\mathbf v)x^u \mathbf y^{\mathbf v}$,
the combinatorial specification, including the boundary condition, translates into the system
of functional equations
\[
\forall\ \mathbf s: \ 
F_{\mathbf s} = [\mathbf p(i_0,\mathbf j_0,0)=\mathbf s\bmod\mathbf m]
+ t \sum_{\mathbf r} ([x^{\geq0}]S_{\mathbf r}^{\mathbf s}) F_{\mathbf r}
+ t \sum_{k\geq 1} \sum_{\mathbf r} x^{-k} ([x^{-k}] S_{\mathbf r}^{\mathbf s}) ([x^{\geq k}] F_{\mathbf r}),
\]
where $[x^{\geq0}]S_{\mathbf r}^{\mathbf s}\in\set K[x,\mathbf y,\mathbf y^{-1}]$ is the Laurent polynomial
obtained from $S_{\mathbf r}^{\mathbf s}$ by discarding all terms with negative exponents in~$x$,
where $[x^{-k}] S_{\mathbf r}^{\mathbf s}\in\set K[\mathbf y,\mathbf y^{-1}]$ is the coefficient of $x^{-k}$ of
$S_{\mathbf r}^{\mathbf s}\in\set K[x,x^{-1},\mathbf y,\mathbf y^{-1}]$, and where
$[x^{\geq k}] F_{\mathbf r}$ denotes the series obtained from $F_{\mathbf r}$ by discarding
all terms $x^i \mathbf y^{\mathbf j} t^n$ with $i<k$.
The functional equation can also be expressed using the $\Delta$ operator defined
by $\Delta F(x,\mathbf y,t) := (F(x,\mathbf y,t)-F(0,\mathbf y,t))/x$.
Note that $[x^{\geq k}] F = x^k\Delta^k F$ for all~$k\in\set N$,
so $x^{-k} ([x^{-k}] S_{\mathbf r}^{\mathbf s}) ([x^{\geq k}] F_{\mathbf r})$
simplifies to $([x^{-k}] S_{\mathbf r}^{\mathbf s})\Delta^k F_{\mathbf r}$.

\begin{example}
  Continuing the previous example and taking all weights to be~$1$,
  we have
  \begin{alignat}1
    F_0(x,y) &= 1 + t (y + y^{-1} + x)F_1(x,y) + t x^{-1} (F_1(x,y) - F_1(0,y)),\label{eq:ex1a} \\
    F_1(x,y) &= t (y + y^{-1} + x)F_0(x,y) + tx(y + y^{-1})F_1(x,y) \notag\\
    &\quad{}+ t x^{-1}(F_0(x,y) - F_0(0,y)) + tx^{-1} (y + y^{-1})(F_1(x,y) - F_1(0, y)).\notag
  \end{alignat}
  Eliminating $F_0(x,y)$ from these equations leads to the equation
  \begin{equation}\label{eq:ex1c}
    \begin{array}{l}
      \bigl(x^2y^2 - t^2(x+y)^2(1+x y)^2 - t x (1+x^2)y(1+y^2)\bigr) F_1(x,y)\\
      =t x y (x+y)(1+x y) - t x y^2 F_0(0,y) - t y (x (1+y^2) + t(x+y)(1+x y)) F_1(0,y).
    \end{array}
  \end{equation}
  which we can solve using the kernel method: the polynomial
  $x^2y^2 - t^2(x+y)^2(1+x y)^2 - t x (1+x^2)y(1+y^2)\in\set Q(y)[[t]][x]$ has two roots in $\overline{\set Q(y)}[[t]]$,
  and if we substitute them for $x$ in equation~\eqref{eq:ex1c}, we get a system of two linear equations for the two unknown
  series $F_0(0,y)$ and $F_1(0,y)$. 
  This system turns out to have a unique solution, which implies that $F_0(0,y)$ and $F_1(0,y)$ are algebraic.
  Consequently, by equation~\eqref{eq:ex1c}, also $F_1(x,y)$ is algebraic.
  Consequently, by equation~\eqref{eq:ex1a}, also $F_0(x,y)$ is algebraic.
  Finally, it follows by algebraic closure properties that $F(x,y)=F_0(x,y)+F_1(x,y)$ is algebraic.
\end{example}

For any particular choice of $\mathbf p$ and $\mathbf m$ and any particular choice of step sets
$\mathbf{S}_{\mathbf{r}}$, we can write down an explicit system of functional equations for the auxiliary
series $F_{\mathbf{r}}$, which we can attempt to solve as illustrated in the example above. 
Potentially, such an attempt could fail, for example because there are too few roots for
applying the kernel method, or because there are too many solutions of the linear system for the
evaluated auxiliary series $F_{\mathbf{r}}(0,\mathbf y)$. The next theorem says that no such problems arise. 

\begin{thm}\label{thm:1}
  The generating function $F(x,\mathbf y,t)\in\set K[x,\mathbf y,\mathbf y^{-1}][[t]]$
  for a model of inhomogeneous lattice walks and a choice of weight functions $w_\mathbf r$
  as specified above is algebraic over $\set K[x,\mathbf y,t]$. 
\end{thm}
\begin{proof}
  As argued above, for every $\mathbf s$ there is a functional equation $F_{\mathbf s}=\cdots$,
  where the right hand side is a $\set K(\mathbf y)[x,t]$-linear combination of $1$ and
  the series $\Delta^i F_{\mathbf r}$. These equations together
  form a system of functional equations which can be written as
  \[
    \mathbf f = \mathbf a + t \sum_{i=0}^k \mathbf B_i \Delta^i\mathbf f,
  \]
  where $\mathbf f$ is the vector of the $F_{\mathbf s}$'s, where
  $\mathbf a$ is a certain explicit vector, and where the $\mathbf B_i$'s are certain explicit matrices
  with entries in $\set K(\mathbf y)[x,t]$.
  According to Theorem~\ref{thm:main} shown in the next section
  (applied with $\set K(\mathbf y)$ in place of~$\set K$),
  the unique solution vector of such a system always has algebraic components.
  This means that every $F_{\mathbf s}$ is algebraic,
  and consequently, the finite sum $F=\sum_{\mathbf r}F_{\mathbf r}$ is algebraic too.
\end{proof}

\section{Systems of Linear Functional Equations}\label{sec:3}

The purpose of the present section is to prove the following theorem,
which says that the solutions of certain systems of linear
functional equations are always algebraic.
Throughout this section, let $\set K$ be a field of characteristic zero.
Recall from the previous section that $\Delta\colon\set K[x][[t]]\to\set K[x][[t]]$
is defined by $\Delta f(x,t) = \frac1x(f(x,t)-f(0,t))$.
Applied to a vector of series, the action of $\Delta$ is meant componentwise. 

\begin{thm}\label{thm:main}
  Let $\mathbf{a}\in\mathbb{K}[x][t]^n$ and $\mathbf{B}_i\in\mathbb{K}[x][t]^{n\times n}$
  ($i=0,\dots,k$). Then the functional equation
  \begin{equation}\label{eq:fun}
    \mathbf{f} = \mathbf{a} + t \sum_{i=0}^k \mathbf{B}_i \Delta^i \mathbf{f}
  \end{equation}
  has a unique solution $\mathbf{f}$ in $\mathbb{K}[x][[t]]^n$, and its components
  are algebraic over $\mathbb{K}[x,t]$.
\end{thm}

It is clear that the functional equation has a unique solution in $\mathbb{K}[x][[t]]^n$,
because we can compute its coefficients recursively via
\[
  [t^0]\mathbf f=\mathbf [t^0] a
  \quad\text{and}\quad
  [t^{n+1}]\mathbf f=[t^{n+1}]\mathbf a+
            \sum_{i=0}^k \sum_{j=0}^n [t^j]\mathbf B_i\Delta^i [t^{n-j}]\mathbf f\quad(n\in\set N).
\]
The nontrivial part of the theorem is that the components of the solution are algebraic.
For proving this part of the theorem, we may assume without loss of generality that $\set K$ is
algebraically closed, and we will do so from now on.

Our proof is an adaption of the proof idea of Thm.~3 in~\cite{bousquet06} to linear systems.
We first bring the unevaluated terms $\mathbf f(x,t)$ hidden in the delta terms to the left hand side,
so that the right hand side only contains $\mathbf f(0,t)$ or other evaluated versions of $\mathbf f$.
This can be done by translating the delta terms into evaluations of partial derivatives. Using
$[x^j]\Delta^i\mathbf f(x,t)=[x^{i+j}]\mathbf f(x,t)$
and
$[x^j]\mathbf f(x,t) = \frac1{j!}\mathbf f^{(j)}(0,t)$,
where $\mathbf f^{(j)}(0,t)$ is the $j$th derivative with respect to~$x$, evaluated at $x=0$, we
can write
\[
 \Delta^i\mathbf f(x,t) = \frac1{x^i}\biggl(\mathbf f(x,t) - \sum_{j=0}^{i-1}\frac{x^j}{j!}\mathbf f^{(j)}(0,t)\biggr).
\]
The functional equation~\eqref{eq:fun} can therefore be rewritten in the form
\begin{equation}\label{eq:fun2}
\biggl(x^k\mathbf I_n - t\sum_{i=0}^k x^{k-i}\mathbf B_i\biggr)\mathbf f(x,t)
= x^k\mathbf a - t \sum_{j=0}^{k-1}\biggl(\sum_{i=j+1}^k\frac{x^{k+j-i}}{j!}\mathbf B_i\biggr)\mathbf f^{(j)}(0,t).
\end{equation}
The key to the proof is the matrix on the left side. We will first prove the theorem
under the additional assumption that this matrix has the form $x^k\mathbf I_n - t\mathbf P$ for some
matrix $\mathbf P$ such that $[t^0x^0]\mathbf P$ is a non-singular diagonal matrix (Lemma~\ref{3}).
Afterwards, the general case is reduced to this situation by a perturbation argument.

In the proof of Lemma~\ref{3}, we will relate eigenvectors of $[t^0x^0]\mathbf P$
to eigenvectors of $\mathbf P$, using the following fact.

\begin{lem}\label{1}
  Let $\mathbf{P}\in\mathbb{K}[x][t]^{n\times n}$, let $\mathbf{P}_0 = [t^0x^0]\mathbf{P}\in\mathbb{K}^{n\times n}$,
  and let $\mathbf{K}= x^k\mathbf{I}_n- t\mathbf{P}\in\set K[x][t]^{n\times n}$, for some $k\in\set N$.
  Let $\lambda$ be an eigenvalue of~$\mathbf{P}_0$, let $\omega$ be a primitive $k$-th root of unity, and let $i\in\{0,\dots,k-1\}$.
  Then there is a series $y(t)=\omega^i\lambda^{1/k}t^{1/k} + \O(t^{2/k})\in\mathbb{K}[[t^{1/k}]]$ such that 
  $\det(\mathbf{K}(y(t),t))=0$. 
  Furthermore, there is a vector $\mathbf{v}(t)\in\mathbb{K}[[t^{1/k}]]^n$ with algebraic coordinates such that
  $\mathbf{v}(t)\mathbf{K}(y(t),t) = 0$ and $\mathbf{v}(0)$ is an eigenvector of $\mathbf{P}_0$ for~$\lambda$.
\end{lem}
\begin{proof}
By definition of the determinant, 
$\det(\mathbf{K}) = \det(x^k \mathbf{I}_n-t\mathbf{P}) = \det(x^k\mathbf{I}_n-t\mathbf{P}_0) + \O(t^2)$.
The polynomial $\det(x^k\mathbf{I}_n-t\mathbf{P}_0) = (-1)^n t^n \det(\mathbf{P}_0 - \frac{x^k}{t}\mathbf{I}_n)\in\set K(t)[x]$
has the root $\omega^i \lambda^{1/k}t^{1/k}$.
Hence, by the Newton-Puiseux algorithm~\cite{kauers10j}, there is a series
$y(t)=\omega^i\lambda^{1/k}t^{1/k}+\O(t^{2/k})\in\set K[[t^{1/k}]]$
such that $\det(\mathbf{K}(y(t),t))=0$.
This means that the matrix $\mathbf{K}(y(t),t)\in\set K(y(t),t)^{n\times n}$ is singular,
so its left-kernel contains a nonzero element $\mathbf{v}\in\set K(y(t),t)^{n\times n}$.
Since $y(t)$ is algebraic (because $\det(\mathbf K)$ is a polynomial in~$x$), also the components of $\mathbf{v}$ are algebraic.
After multiplying by a suitable power of~$t$, if needed, we may assume
that the components of $\mathbf{v}(t)$ are in $\set K[[t^{1/k}]]$ and that $\mathbf{v}(0)$ is not
the zero vector. Then $\mathbf{v}(t)\mathbf{K}(y(t),t)=0$ implies $\mathbf{v}(0)(\lambda\mathbf{I}_n-\mathbf{P}_0)=0$,
which completes the proof of the lemma.
\end{proof}

Secondly, we will have to ensure that a certain matrix is nonsingular.
This fact is provided by the following lemma.

\begin{lem}\label{2}
Let $\lambda_0,\dots,\lambda_{n-1}\in\mathbb{K}\setminus\{0\}$ and let $\omega$ be a primitive $k$-th root of unity. For $u,v=0,\dots,nk-1$ define $c_{u,v}=(\omega^{u \bmod k}\lambda_{\lfloor u/k \rfloor})^{\lfloor v/n\rfloor} \delta_{\lfloor u/k \rfloor,v \bmod n}$. Then 
\[
\det((c_{u,v})_{u,v=0}^{nk-1}) = \pm\Bigl(\prod_{0\leq i<j<k}(\omega^j - \omega^i)\Bigr)^n \prod_{\ell=0}^{n-1}\lambda_\ell^{\binom{k}{2}}. 
\]
In particular, this determinant is not zero.
\end{lem}
\begin{proof}
The expression on the right is non-zero because $\lambda_\ell\neq 0$ for all $\ell$ and $\omega$ is a primitive root of unity. It remains to prove the claimed identity.

We permute the columns of the matrix such that the entry at position $(u,v)$ is 
\[
 (\omega^{u \bmod k}\lambda_{\lfloor u/k \rfloor})^{v \bmod k} \delta_{\lfloor u/k \rfloor,\lfloor v/k \rfloor}.
\]
The resulting matrix is block diagonal with $n$ blocks of size $k\times k$. The $\ell$-th block is 
the Vandermonde matrix $((\omega^{ij}\lambda_\ell^j))_{i,j=0}^{k-1}$, whose determinant is $\lambda_{\ell}^{\binom{k}{2}}\prod_{i<j}(\omega^j-\omega^i)$. Since the determinant of a block diagonal matrix is the product of the determinants of its blocks,
we arrive at the desired conclusion.
\end{proof}

The idea for the proof of Lemma~\ref{3} is now to replace the variable $x$ by various algebraic
series $y(t)$ in such a way that the terms $\mathbf f(x,t)$ in \eqref{eq:fun2} drop out and a linear
system for the components of $\mathbf f^{(j)}(0,t)$ arises, and then to show that this system has a unique solution.

\begin{lem}\label{3}
  Let $\lambda_0,\dots,\lambda_{n-1}\in\set K\setminus\{0\}$ be pairwise distinct, and let $\mathbf{E}=\diag(\lambda_0,\dots,\lambda_{n-1})\in\set K^{n\times n}$. 
  Let $\mathbf{a}\in\set K[x][t]^n$, and let $\mathbf{P},\mathbf{Q}_0,\dots,\mathbf{Q}_{k-1}\in\set K[x][t]^{n\times n}$.
  Suppose that $\mathbf{P}=\mathbf{E}+\O(t)$ and $\mathbf{Q}_j=x^j \mathbf{E} + \O(t)$ for $j=0,\dots,k-1$.
  If $\mathbf{f}\in\set K[x][[t]]^n$ and $\mathbf{g}_0,\dots,\mathbf{g}_{k-1}\in\set K[[t]]^n$ are such that 
  \begin{alignat}1\label{funeq}
    \left(x^k \mathbf{I}_n - t \mathbf{P}\right)\mathbf{f} = x^k \mathbf{a} - t\sum_{j=0}^{k-1} \mathbf{Q}_j\mathbf{g}_j,
  \end{alignat}
  then the components of $\mathbf{g}_0,\dots,\mathbf{g}_{k-1}$ and $\mathbf{f}$ are algebraic over $\set K[x][t]$.
\end{lem}

\begin{proof}
  By Lemma~\ref{1}, the polynomial $\det(x^k\mathbf{I}_n-t\mathbf{P})$ has $nk$
  series roots $y_{ij}(t)$ of the form $y_{ij}(t) =
  \omega^i\lambda_j^{1/k}t^{1/k} + \O(t^{2/k})$.  For each such root, the matrix
  $y_{ij}(t)^k\mathbf{I}_n-t\mathbf{P}(y_{ij}(t),t)$ is singular, and, again
  using Lemma \ref{1}, the left-kernel of
  $y_{ij}(t)^k\mathbf{I}_n-t\mathbf{P}(y_{ij}(t),t)$ contains a vector
  $\mathbf{v}_{ij}(t)=\lambda_j^{-1}\mathbf{e}_j + \O(t^{1/k})$, whose
  coordinates are algebraic.
  Here $\mathbf e_j$ denotes the $j$th unit vector. 

  For $i=0,\dots,k-1$ and $j=0,\dots,n-1$ we replace $x$ by $y_{ij}(t)$ in
  equation \eqref{funeq} and multiply with $\mathbf{v}_{ij}(t)$ from the
  left. This gives an inhomogeneous linear system with $nk$ equations for the $nk$ unknown
  components of $\mathbf{g}_0(t),\dots,\mathbf{g}_{k-1}(t)\in\mathbb{K}[[t]]^n$,
  whose coefficient matrix is
  \[
  \begin{pmatrix}
    \mathbf{v}_{0,0}(t) \mathbf{Q}_0(y_{0,0}(t),t) & \cdots & \mathbf{v}_{0,0}(t) \mathbf{Q}_{k-1}(y_{0,0}(t),t) \\
    \mathbf{v}_{1,0}(t) \mathbf{Q}_0(y_{1,0}(t),t) & \cdots & \mathbf{v}_{1,0}(t) \mathbf{Q}_{k-1}(y_{1,0}(t),t) \\
    \vdots & \ddots & \vdots \\
    \mathbf{v}_{k-1,n-1}(t) \mathbf{Q}_0(y_{k-1,n-1}(t),t) & \cdots & \mathbf{v}_{k-1,n-1}(t) \mathbf{Q}_{k-1}(y_{k-1,n-1}(t),t) 
  \end{pmatrix}.
  \]
  We are done if we can show that this matrix is invertible, because this
  implies that the inhomogeneous linear system for the components of
  $\mathbf{g}_0,\dots,\mathbf{g}_{k-1}$ has a unique solution. The components of its
  solution vector must be algebraic, because all the series appearing in the
  linear system are algebraic.  Returning to equation~\eqref{funeq}, we finally
  see that the algebraicity of the components of
  $\mathbf{g}_0,\dots,\mathbf{g}_{k-1}$ implies the claimed algebraicity of the
  components of $\mathbf{f}$.

  To see that the matrix above is invertible, we use the assumption that
  $\mathbf{Q}_\ell(y_{ij}(t),t) = \mathbf{E}\hspace{1 pt} y_{ij}(t)^\ell + \O(t) = (\omega^i\lambda_j^{1/k})^\ell \mathbf{E} t^{\ell/k} + \O(t^{(\ell+1)/k}) \in\set K[[t^{1/k}]]^{n\times n}$.
  Together with $\mathbf{e}_j\mathbf{E}=\lambda_j \mathbf{e}_j$, where $\mathbf e_j$ is again the $j$th unit vector, it follows that
  \[
    \mathbf{v}_{ij}(t) \mathbf{Q}_\ell(y_{ij}(t),t) =  (\omega^i\lambda_j^{1/k})^\ell \mathbf{e}_j t^{\ell/k} + \O(t^{(\ell+1)/k})\in\set K[[t^{1/k}]]^n,
  \]
  for $i=0,\dots,k-1$, $j=0,\dots,n-1$ and $\ell=0,\dots,k-1$.
  Therefore, for $u=0,\dots,nk-1$ and $v=0,\dots,nk-1$, the entry of the matrix at position $(u,v)$
  is  
  \[
  c_{u,v}:=(\omega^{u\bmod k}\lambda_{\lfloor u/k\rfloor}^{1/k})^{\lfloor v/n\rfloor}t^{\lfloor v/n \rfloor /k}\delta_{\lfloor u/k\rfloor,v\bmod n} + \O(t^{(\lfloor v/n \rfloor + 1) /k}).
  \]
  By Lemma~\ref{2}, we have $\det((c_{u,v})_{u,v=0}^{nk-1})\neq0$.
  This completes the proof.
\end{proof}

\begin{proof}[Proof of Thm.~\ref{thm:main}]
  We have already argued that existence and uniqueness of a solution in $\set K[x][[t]]^n$ are evident.
  To show that its components are algebraic, we bring equation \eqref{eq:fun} into a form
  where Lemma~\ref{3} applies.
  Let $\epsilon$ be a new variable, 
  let $\lambda_1,\dots,\lambda_n\in\set K\setminus\{0\}$ be pairwise distinct
  and set $\mathbf E=\epsilon\diag(\lambda_1,\dots,\lambda_n)$.
  Set $\tilde{\mathbf a}(x,t):=\mathbf{a}(x,t^2)$,
  $\tilde{\mathbf B}_i(x,t):=t\mathbf{B}_i(x,t^2)$ ($i=0,\dots,k-1$), and
  $\tilde{\mathbf B}_k(x,t):=\mathbf{E} + t\mathbf{B}_k(x,t^2)$, and
  consider the system
  \[
    \tilde{\mathbf f} = \tilde{\mathbf a} + t\sum_{i=0}^k\tilde{\mathbf B}_i\Delta^i\tilde{\mathbf f}. 
  \]
  This system has a unique solution $\tilde{\mathbf f}\in\set K[x,\epsilon][[t]]^n$,
  which is related to the solution $\mathbf f$ of the original equation~\eqref{eq:fun}
  via $\mathbf f(x,t^2)=[\epsilon^0]\tilde{\mathbf f}(x,t)$. 
  We are done if we can show that the components of $\tilde{\mathbf f}$ are algebraic, because then so are
  the components of $\mathbf f$.

  Indeed, translating the delta terms to ordinary derivatives, as earlier, gives 
  \[
   \biggl(x^k\mathbf I_n - t\sum_{i=0}^k x^{k-i}\tilde{\mathbf B}_i\biggr)\tilde{\mathbf f}(x,t)
   = x^k\tilde{\mathbf a} -
   t \sum_{j=0}^{k-1}\biggl(\sum_{i=j+1}^k\frac{x^{k+j-i}}{j!}\tilde{\mathbf B}_i\biggr)\tilde{\mathbf f}{\vphantom{\mathbf f}}^{(j)}(0,t).
  \]
  For the matrix $\mathbf P=\sum_{i=0}^k x^{k-i}\tilde{\mathbf{B}}_i\in\set K(\epsilon)[x,t]^{n\times n}$
  we have $\mathbf P=\mathbf E+\O(t)$,
  and for the matrices $\mathbf Q_j=k!\sum_{i=j+1}^k\frac{x^{k+j-i}}{j!}\tilde{\mathbf B}_i\in\set K(\epsilon)[x,t]^{n\times n}$
  we have $\mathbf Q_j=x^j\mathbf E+\O(t)$ for $j=0,\dots,k-1$.   
  Therefore, Lemma~\ref{3} applies to the perturbed equation above and yields the desired algebraicity result.
  (The lemma is applied with $\set K$ replaced by some algebraic closure of $\set K(\epsilon)$
  and with $\tilde{\mathbf f}^{(j)}(0,t)/k!$ in the role of~$\mathbf g_j$.)
\end{proof}

\section{Models with more than one restriction}\label{sec:4a}

We have seen that inhomogeneous models in a half-space $\mathbb{Z}_{\geq 0}\times\mathbb{Z}^{d-1}$ always have an
algebraic generating function. More generally, we can consider walks restricted to $\set Z_{\geq 0}^p\times\set Z^q$.
In this case, the question arises whether some of the $p$ constraints are
implied by the others, which has led Bostan et al.~\cite{bostan16a} to introduce the \emph{dimension} of a model. Here 
we generalize this notion to inhomogeneous walks.

First consider unrestricted models in~$\set Z^d$. Fix $k$, $\mathbf m$, $\mathbf p$ and a collection of step
sets $\mathbf S_{\mathbf r}\subseteq\set Z^d$ like in Section~\ref{sec:2}. 
Let $\mathbf{S}$ be the union of some disjoint copies of the sets $\mathbf{S}_{\mathbf{r}}$, so that 
a walk in $\set Z^d$ of length $n$ can be viewed as a word $\omega$ over the alphabet~$\mathbf S$.
To any such walk~$\omega$, we associate the vector $(a_{\mathbf u})_{\mathbf u\in\mathbf S}$
where $a_{\mathbf u}\in\set N$ is the number of occurrences of $\mathbf u$ in~$\omega$. While for unrestricted homogeneous
models, every vector of natural numbers is associated with some walk, this is no longer true for inhomogeneous
models. For example, for space-inhomogeneous walks in $\set Z^2$ with $\mathbf S_0=\diagr{NE}$ and $\mathbf S_1=\diagr{SW}$,
the vector $(1,1)$ is not associated with a walk. The next lemma is a characterization of the vectors that are associated with walks.

\begin{lem}\label{lem:4}
  Let $k$, $\mathbf m$, $\mathbf p$ and $\mathbf S_{\mathbf r}\subseteq\set Z^d$ be as in Section~\ref{sec:2},
  and let $\mathbf i_0\in\set Z^d$.
  Let $\mathbf{S}$ be a disjoint union of the sets $\mathbf{S}_{\mathbf r}$ and $(a_{\mathbf u})_{\mathbf u\in\mathbf S}$
  be a vector of nonnegative integers.
  Let $G$ be the multi-graph with all the $\mathbf r$'s as vertices and
  with $\sum_{\mathbf u\in\mathbf S_{\mathbf r}^{\mathbf s}}a_{\mathbf u}$ edges from $\mathbf r$ to~$\mathbf s$, for all
  vertices $\mathbf r,\mathbf s$. 
  Then $(a_{\mathbf u})_{\mathbf u\in\mathbf S}$ is associated with a walk starting at $\mathbf i_0$
  if and only if $G$ has an Eulerian path starting at a vertex $\mathbf r_0$
  with $\mathbf p(\mathbf i_0,0)=\mathbf r_0\bmod\mathbf m$. 
\end{lem}
\begin{proof}
  Since $\mathbf S$ is the union of disjoint copies of the step sets $\mathbf S_{\mathbf r}$,
  every step $\mathbf u\in\mathbf S$ belongs to exactly one such set, and in particular to exactly one subset
  $\mathbf S_{\mathbf r}^{\mathbf s}$. Therefore, any walk $\omega$ in the model which has
  $(a_{\mathbf u})_{\mathbf u\in\mathbf S}$
  as associated vector can be translated into a path in the graph with a starting vertex as required which
  uses $\sum_{\mathbf u\in\mathbf S_{\mathbf r}^{\mathbf s}}a_{\mathbf u}$ times a step from $\mathbf r$ to
  $\mathbf s$, for all vertices $\mathbf r,\mathbf s$. This is an Eulerian path. 
  Conversely, let $(\mathbf r_0,\mathbf r_1,\dots,\mathbf r_n)$ be an Eulerian path in~$G$ with $\mathbf r_0$
  as stated in the lemma. Then for each pair $\mathbf r,\mathbf s$ of vertices there are
  $\sum_{\mathbf u\in\mathbf S_{\mathbf r}^{\mathbf s}}a_{\mathbf u}$ many indices $i\in\{1,\dots,n\}$ such that
  $(\mathbf r_{i-1},\mathbf r_i)=(\mathbf r,\mathbf s)$. For every $i\in\{1,\dots,n\}$, assign an arbitrary
  step $\mathbf u\in\mathbf S_{\mathbf r_{i-1}}^{\mathbf r_i}$ in such a way that each step $\mathbf u$ is chosen
  exactly $a_{\mathbf u}$ times,
  and let $\omega$ be the walk composed of the selected steps. 
  By definition of the sets $\mathbf S_{\mathbf r}^{\mathbf s}$, the walk $\omega$ belongs to the model.
\end{proof}

The condition for a graph to have an Eulerian path can be encoded as a system of linear constraints on the in-degrees and
out-degrees of the vertices of the graph~\cite{bollobas98}. In our case these are linear equations for the
variables~$a_{\mathbf u}$.
It can also be encoded into linear equations that the Eulerian path should start or end at prescribed vertices. 

A walk $\omega$ ends in $\mathbb{Z}_{\geq 0}^p\times \mathbb{Z}^q$ iff for its associated vector
$(a_{\mathbf u})_{\mathbf u\in\mathbf S}$ we have
\begin{equation}\label{eq:10}
\sum_{(u_1,\dots,u_{p+q})\in\mathbf S} a_s u_i\geq 0\quad\text{for all $i=1,\dots,p$,}
\end{equation}
and it stays entirely in $\mathbb{Z}_{\geq 0}^p\times\mathbb{Z}^q$ iff these inequalities hold for all prefixes
of~$\omega$.
In Def.~2 of~\cite{bostan16a}, the dimension of a homogeneous model for $\mathbb{Z}_{\geq 0}^p\times \mathbb{Z}^q$ was
defined as the smallest $\delta\in\set N$ such that there are $\delta$ inequalities in \eqref{eq:10} which imply all the others.
In view of Lemma~\ref{lem:4}, we define more generally the dimension of an inhomogeneous models for
$\mathbb{Z}_{\geq 0}^p\times\mathbb{Z}^q$ as the smallest $\delta\in\set N$ such that there are $\delta$ inequalities in~\eqref{eq:10}
which together with the linear equations encoding the existence of an Eulerian path in the graph $G$ of Lemma~\ref{lem:4}
imply the remaining inequalities. 
Like in~\cite{bostan16a}, the dimension of a model can now be determined by linear programming.

\section{Inhomogeneous Lattice Walks in the Quarter Plane}\label{sec:4}

We have no satisfactory theory for inhomogeneous models for the quarter plane.
However, for all purely time-inhomogeneous and all purely space-inhomogeneous models
with $k=1$ and small steps, we have produced an experimental classification which is 
available at \url{http://www.algebra.uni-linz.ac.at/people/mkauers/inhomogeneous/}.
Up to symmetry, there are 32993 pairs $\mathbf S_0,\mathbf S_1\subseteq\diagr{N,E,S,W,NE,SE,NW,SW}$.
Removing trivial cases (whose counting sequence is ultimately constant),
zero- and one-dimensional cases (whose generating function is algebraic by Thm.~\ref{thm:1}),
and homogeneous cases (whose nature is known)
leaves us with 23906 space-inhomogeneous and 25370 time-inhomogeneous cases.

For each of these, we computed the first 10000 terms (modulo the prime 45007)
of the generating function $F(1,1,t)$ and tried to guess a differential equation.
When an equation was found, we also searched for an algebraic equation.
These computations took altogether about 30 years of computing time. 
As a result, 3784 space-inhomogeneous models and 2603 time-inhomogeneous models seem to be D-finite,
including 2474 and 1535 seemingly algebraic cases, respectively.
For space-inhomogeneous models, the largest differential equations we found have order 24 and degree~1183,
such an equation appears for example for $\mathbf S_0=\diagr{SW,N,NE}$ and $\mathbf S_1=\diagr{SW,S,NE,E,SE}$. 
For time-inhomogeneous models, the largest differential equation we found has order 28 and degree~1256
and appears for $\mathbf S_0=\diagr{N,S,W,E}$, $\mathbf S_1=\diagr{W,N,E,SE,SW}$.
In view of these sizes, it is likely that some further D-finite models could be discovered with more terms.

The techniques of~\cite{bousquet10} for proving D-finiteness seem to apply only to a very limited number of
cases. We conclude with two examples where they work and invite our readers to find proofs for further
conjecturally D-finite cases. 

\begin{example}
  Consider the time-inhomogeneous model in $\set Z_{\geq0}^2$ with $\mathbf{S}_0 = \diagr{NE,NW,S}$ and
  $\mathbf{S}_1 = \diagr{N,W,E,SW,SE}$. Let $S_0,S_1\in\mathbb{Q}[x,x^{-1},y,y^{-1}]$ be the corresponding Laurent polynomials,
and let $F_0(x,y),F_1(x,y)\in\mathbb{Q}[x,y][[t]]$ be the power series
counting the number of walks of even and odd lengths, respectively. Then $F(x,y)
= F_0(x,y) + F_1(x,y)$ is the generating function of the model.
 The functional equations for $F_0(x,y)$ and $F_1(x,y)$ are
\begin{align*}
F_0(x,y) &= 1 + t S_1 F_1(x,y) - t([y^{<0}]S_1)F_1(x,0)- t([x^{<0}]S_1)F_1(0,y) + t([x^{<0}y^{<0}])F_1(0,0)\\
 F_1(x,y) &= tS_0 F_0(x,y) - t([y^{<0}]S_0)F_0(x,0).
\end{align*}
  Following Bousquet-M\'elou and Mishna~\cite{bousquet10}, we consider
  the groups $G_0$ and $G_1$ generated by the rational maps $\Phi_0\colon (x,y)\mapsto (\frac1x,y)$ and
  $\Psi_0\colon (x,y)\mapsto(x,\frac{1}{y(x+\frac1x)})$, 
  and $\Phi_1\colon (x,y)\mapsto (\frac1x,y)$ and
  $\Psi_1\colon (x,y)\mapsto(x,\frac1y (x+\frac1x))$, respectively. 
  We multiply the two equations above by $xy$ and take the so-called
  orbit sum of the first equation with respect to $G_1$, and of the second one with respect to~$G_0$.
  This eliminates all terms $F_{\mathbf r}(x,0)$, $F_{\mathbf r}(0,y)$ and $F_{\mathbf{r}}(0,0)$
  with $\mathbf{r}\in\{0,1\}$ and leaves us with
  \begin{alignat*}1
    \sum_{g\in G_1} \sgn(g) g(xyF_0(x,y)) &= \sum_{g\in G_1} \sgn(g) g(xy) + t S_1 \sum_{g\in G_1} \sgn(g) g(xyF_1(x,y))\\
    \sum_{g\in G_0} \sgn(g) g(xyF_1(x,y)) &= t S_0 \sum_{g\in G_0} \sgn(g) g(xyF_0(x,y)).
  \end{alignat*}
  It is easy to check that replacing $y$ by $\frac{1}{y}(x+\frac1x)$ in the second equation gives
  \[
   \sum_{g\in G_1} \sgn(g) g(xyF_1(x,y)) = t S_0(x,\frac{1}{y}(x+\frac1x)) \sum_{g\in G_1} \sgn(g) g(xyF_0(x,y)).
  \]
  From this equation and the first of the two previous two equations we get
    \begin{equation*}
    \sum_{g \in G_1} \sgn(g) g(xyF_0(x,y)) = 
    \frac{1}{1-t^2 S_0(x,\frac{1}{y}(x+\frac1x)) S_1} \sum_{g\in G_1} \sgn(g)g(xy)
  \end{equation*}
  Extracting the positive part gives
  \begin{equation*}
    F_0(x,y) = \frac{1}{xy}[x^{>0}y^{>0}]  \frac{1}{1-t^2 S_0(x,\frac{1}{y}(x+\frac1x)) S_1} \sum_{g\in G_1} \sgn(g) g(xy).
  \end{equation*}
  This expression implies D-finiteness of $F_0(x,y)$, and back-substituting into the earlier equations and using
  D-finite closure properties gives the D-finiteness of $F_1(x,y)$ and of $F(x,y)$.
\end{example}

\begin{example}
  Consider the space-inhomogeneous model in $\set Z_{\geq0}^2$
  with $\mathbf{S}_0=\diagr{N,S,E,W}$
  and $\mathbf{S}_1=\diagr{N,S,E,W,SE,SW,NE,NW}$
  studied by D'Arco et al.~\cite{darco16}.
  Define $\mathbf{S}_{\mathbf r}^{\mathbf s}\subseteq\mathbf{S}_{\mathbf r}$ for $\mathbf r,\mathbf s\in\{0,1\}$
  as in Section~\ref{sec:2},
  and let $S_{\mathbf r}^{\mathbf s}\in\set Q[x,x^{-1},y,y^{-1}]$ be the corresponding Laurent polynomials. 
  Write $F_{\mathbf r}(x,y)$ for $\mathbf r\in\{0,1\}$ for the power series that counts the number of walks
  ending at points $(i,j)$ with $i+j=\mathbf r\bmod 2$. 
  The functional equations for $F_0(x,y)$ and $F_1(x,y)$ are
  \begin{alignat*}1
     F_0(x,y) &= 1 + t S_1^0 F_1(x,y) - t([y^{<0}]S_1^0) F_1(x,0) - t([x^{<0}]S_1^0) F_1(0,y)\\ 
     F_1(x,y) &= t S_0^1 F_0(x,y) - t ([y^{<0}]S_0^1) F_0(x,0) - t([x^{<1}]S_0^1) F_0(0,y) \\
     &\quad{} + t S_1^1 F_1(x,y) - t ([y^{<0}]S_1^1) F_1(x,0) - t([x^{<1}]S_1^1) F_1(0,y).
  \end{alignat*}
  Consider the group $G$ generated by the rational maps $\Phi\colon (x,y)\mapsto (x,\frac1y)$ and
  $\Psi\colon (x,y)\mapsto(\frac1x,y)$.
  Like in the previous example, multiply both equations by $xy$ and take the so-called
  orbit sum for this group.
  This eliminates all terms $F_{\mathbf r}(x,0)$ and $F_{\mathbf r}(0,y)$ and leaves us with
  \begin{alignat*}2
    \sum_{g\in G} \sgn(g) g(xyF_0(x,y)) &= \sum_{g\in G} \sgn(g) g(xy) + tS_1^0 \sum_{g\in G} \sgn(g) g(xyF_1(x,y))\\
    \sum_{g\in G} \sgn(g) g(xyF_1(x,y)) &= tS_0^1 \sum_{g\in G} \sgn(g) g(xyF_0(x,y)) + tS_1^1 \sum_{g\in G} \sgn(g) g(xyF_1(x,y)).
  \end{alignat*}
  From those equations we deduce 
  \begin{equation*}
    \sum_{g \in G} \sgn(g) g(xyF_1(x,y)) = \frac{tS_0^1}{1-tS_1^1-t^2S_0^1S_1^0} \sum_{g\in G} \sgn(g)g(xy),
  \end{equation*}
  and extracting the positive part gives
  \begin{equation*}
    F_1(x,y) = \frac1{xy}[x^{>0}y^{>0}] \frac{tS_0}{1-tS_1^1-t^2S_0^1S_1^0} \sum_{g\in G} \sgn(g) g(xy).
  \end{equation*}
  This expression implies D-finiteness of $F_1(x,y)$, and back-substituting into the earlier equations
  and using D-finite closure properties gives the D-finiteness of $F_0(x,y)$ and of the full
  generating function $F(x,y)=F_0(x,y)+F_1(x,y)$.
 \end{example}

\printbibliography

\end{document}